\documentclass[11pt]{amsart}
\usepackage{amsmath, amssymb, amsthm, a4wide}
\usepackage{geometry}
\usepackage{fancyhdr}
\usepackage{enumitem}
\newtheorem{theorem}{Theorem}[section]

\newtheorem{definition}[theorem]{Definition}
\newtheorem{proposition}[theorem]{Proposition}

\newtheorem{remark}[theorem]{Remark}

\newtheorem{corollary}[theorem]{Corollary}
\numberwithin{equation}{section}

\def\cp{\mathcal{P}}

\def\cw{\mathcal{W}}

\def\be{\mathbb{E}}

\def\nn{{n\rightarrow\infty}}

\def\vp{\varphi}
\def\ve{\varepsilon}

\begin{document}

\title[Generalized Wasserstein distance and weak convergence of sublinear expectations]{Generalized Wasserstein distance \\and weak convergence of sublinear expectations}
\author{Xinpeng Li}
\author{Yiqing Lin}
\address{Institute for Advanced Research and School of Mathematics\newline\indent
Shandong University\newline\indent
250100, Jinan, China}
\email{lixinpeng@sdu.edu.cn}
\address{Fakult\"{a}t f\"{u}r Mathematik\newline\indent
 Universit\"{a}t Wien\newline\indent
1090 Wien, Austria}
\email{yiqing.lin@univie.ac.at}

%\date{\today}



\begin{abstract}
In this paper, we define the generalized Wasserstein distance for sets of Borel probability measures and demonstrate that the weak convergence of sublinear expectations can be characterized by means of this distance.
\end{abstract}
\maketitle
\section{Introduction}
\noindent In classical probability theory, the Wasserstein distance of order $p$ between two Borel probability measures $\mu$ and $\nu$ in the Polish metric space $(\Omega,d)$ is defined by
$$W_p(\mu,\nu)=\inf\{E[d(X,Y)^p]^{\frac{1}{p}}:\ \text{law}(X)=\mu, \ \text{law(Y)}=\nu\},$$
which is related to optimal transport problems and can also be used to characterize the weak convergence of probability measures in the Wasserstein space $P_p({\Omega})$ (see Definition 6.4 in Villani \cite{villani2}), roughly speaking, under some moment condition, that $\mu_k$ converges weakly to $\mu$ is equivalent to $W_p(\mu_k,\mu)\rightarrow 0$. In particular, $W_1$ is commonly called Kantorovich-Rubinstein distance and the following duality formula is well-known:
\begin{equation}\label{dua}
W_1(\mu,\nu)=\sup_{||\vp||_{\rm{Lip}\leq 1}}\{E_{\mu}[\vp]-E_{\nu}[\vp]\},
\end{equation}
where $\Phi:=\{\varphi: ||\vp||_{\rm{Lip}}\leq 1\}$ is the collection of all Lipschitz functions on $(\Omega, d)$ whose Lipschitz constants are at most 1 (see Villani \cite{villani1,villani2}).\\

Recently, Peng  establishes a theory of time-consistent sublinear expectations in \cite{Peng2008b, pengbook}, i.e.,  the $G$-expectation theory. % has been
%estab-
%lished in [15, Peng2004] and [16, Peng2005]
%In this new framework related stochastic analysis tools is well developed and
This new theory provides with several important models and tools for studying robust hedging and utility maximization problems in the financial markets under Knightian uncertainty (see eg. Epstein and Ji \cite{EJ} and Vorbrink \cite{Vor}). According to Denis et al. \cite{DHP}, $G$-expectation can be expressed as an upper expectation introduced in Huber \cite{Huber},
i.e., $
\be_G[\cdot]:=\sup_{\mu\in\cp}E_\mu[\cdot]
$,
where $\mathcal{P}$ is a weakly compact set of martingale measures. Compared with the sublinear functionals studied by Lebedev in \cite{Le}, $G$-expectation is non-dominated due to the mutual singularity of elements in $\mathcal{P}$.\\

% $G$-expectation is no longer a dominated sublinear expectation, which is in \\

In this paper, the classical notion of distance between two probability measures is extended, more precisely, we define a generalized Wasserstein metric for two weakly compact sets of probability measures, which could be regarded as generators of (not necessarily dominated) sublinear expectations. This metric is given by a Hausdorff type distance function, whose value is the ``greatest'' of all Wasserstein distances from a probability measure in one set to the ``closest'' probability measure in the other set. Furthermore, we adopt the notion of weak convergence for sublinear expectations from Peng \cite{Peng2008b,pengbook,p3} and then characterize this type of convergence by the generalized Wasserstein metric. We notice that the main techniques in this paper could be applied to considering transport inequality in the sublinear expectation context, for example, the transport-entropy inequality and the related large deviation principle. Also, the results may provide with some useful tools for the further study of robust optimal transport problems on sublinear expectation spaces.\\
% can be naturally used to .
%weak convergence of sublinear expectations.  when he studied the law of large numbers and central limit theorems on the sublinear expectation under new i.i.d. assumptions.   We believe that some properties of this paper will become important and basic tools in the further development of sublinear expectation theory, for example, the transport-entropy inequality and related large deviation principle on the sublinear expectation space.\\

This paper is organized as follows: in the next section, we recall some notations in the framework of sublinear expectations and define the generalized Wasserstein distance. In Section 3, the duality formula for generalized Kantorovich-Rubinstein distance are given and characterization of the weak convergence of sublinear expectations is discussed.

\section{Preliminaries}
\noindent We first recall preliminaries in the framework of sublinear expectations.
Let $(\Omega,d)$ be a Polish space equipped with the Borel $\sigma$-algebra $\mathcal{B}(\Omega)$. In the sequel, we only consider the Borel probability measures on this space and denote by $\mathcal{P}(\Omega)$ the collection of all sets of these measures. For each $\mathcal{P}\in \mathcal{P}(\Omega)$, one can define a sublinear expectation $\be^\cp[\cdot]$ in the following form:
\begin{equation}\label{sub}
\be^\cp[\vp]:=\sup_{\mu\in\cp}E_\mu[\vp],\ \forall \vp\in \mathcal{L}^0(\Omega),
\end{equation}
where $\mathcal{L}^0(\Omega)$ is the space of all $\mathcal{B}(\Omega)$-measurable real functions.\\

\indent It is easy to check that $\mathbb{E}^\mathcal{P}[\cdot]$ satisfies the following properties: for $X$, $Y\in \mathcal{L}^0(\Omega)$,\\[6pt]
\noindent{\textup{(1)} {\bf Monotonicity:}\ \ } If $X\geq Y$, then $\be^\cp[X]\geq \be^\cp[Y];$\newline
{\textup{(2)} {\bf Constant preserving:} \ \ $\be^\cp[c]=c$, $\forall c\in\mathbb{R}$;
\newline
{\textup{(3)} {\bf Sub-additivity: \ \ } $\be^\cp[X+Y]\leq
\be^\cp[X]+\mathbb{{E}}^\mathcal{P}[Y];$
\newline{\textup{(4)} {\bf Positive
homogeneity:} } \ $\be^\cp[\lambda X]=\lambda \be^\cp[X]$,$\  \  \forall \lambda\geq 0$,\\[6pt]
\noindent where the usual rule $0\cdot\infty=0$ is applied in (4). \\[6pt]
\indent In Peng \cite{Peng2008b,pengbook}, the properties of sublinear expectations are systematically studied and some limit theorems are proved. In particular, Peng \cite{p3} introduces the following notion of weak convergence for sublinear expectations:
\begin{definition}\label{ch4d10}
We say that a sequence $\{\be^{\cp_n}[\cdot]\}_{n=1}^\infty$ of sublinear expectations converges weakly to some $\be^\cp[\cdot]$, if for each bounded and continuous function $\varphi\in C_{b}(\Omega)$,
$$\lim_{\nn}\be^{\cp_n}[\vp]=\be^{\cp}[\vp].$$
\end{definition}
\indent Now we define the generalized Wasserstein metric as follows:
\begin{definition}\label{d2}
Let $\cp_1$, $\cp_2\in \mathcal{P}(\Omega)$. For $p\geq 1$, we define the $p$-order Wasserstein metric between $\cp_1$ and $\cp_2$ in the following form:
$$\mathcal{W}_p(\cp_1,\cp_2)=\max\bigg(\sup_{\mu\in\cp_1}\inf_{\nu\in\cp_2}W_p(\mu,\nu),\sup_{\nu\in\cp_2}\inf_{\mu\in\cp_1}W_p(\mu,\nu)\bigg),$$
where ${W}_p(\mu,\nu)$ is the classical $p$-order Wasserstein distance between two probability measures $\mu$ and $\nu$.
\end{definition}
\indent It is evident that the function $\cw_p$ is ordered in $p$ as its classical counterpart $W_p$, namely, $\cw_p\leq \cw_q$, if $1\leq p\leq q$.
\begin{remark}
It is easy to see that $\mathcal{W}_p$ is a semi-metric and it will be a metric if we only consider weakly compact elements in $\cp(\Omega)$.

%satisfies the axioms of a distance on the subspace $\tilde{\mathcal{P}}(\Omega) \subset \mathcal{P}(\Omega)$, which consists of weakly compact sets %$\mathcal{P}$, however, it may take the value $\infty$.\\

%Notice that if the compactness is violated, then $\mathcal{W}_p$ may fail to be a distance. For instance, consider $\Omega=\{0, 1\}$ and two sets of probabilities $\mathcal{P}$ and $\mathcal{P}'$:
%$$
%\mathcal{P}:=\{(1-x)\delta_0+x\delta_1: x\in (0, 1)\};\ \mathcal{P}':=\{(1-x)\delta_0+x\delta_1: x\in [0, 1]\},
%$$
%which are two different sets, however $\mathcal{W}_1(\mathcal{P}, \mathcal{P}')=0$.
\end{remark}
Similarly to Definition 6.4 in Villani \cite{villani2}, we can define a space $\mathcal{P}_p(\Omega)$, on which $\mathcal{W}_p$ defines a finite distance.
\begin{definition} (Wasserstein space)\label{def24}
For $p\geq 1$, we denote by $\mathcal{P}_p(\Omega)$ a subset of $\mathcal{P}(\Omega)$, which is the collection of all $\mathcal{P}$ such that

(a) $\cp$ is weakly compact and convex;

(b) For an arbitrary point $\omega_0\in\Omega$,
$$
\lim_{K\rightarrow\infty}\be^\cp[d(\omega_0,\cdot)^p\mathbf{1}_{\{\omega\in\Omega:\ d(\omega_0,\omega)\geq K\}}(\cdot)]=0.
$$
\end{definition}

\begin{remark}
\ \

\noindent(i) It is easy to verify that the space $\cp_p(\Omega)$ does not depend on the choice of $\omega_0$.

\noindent(ii) The assumptions for the space $\cp_p(\Omega)$ are crucial for our main result. In particular, both assumption (a) and (b) enable us to apply Sion's result \cite{S} and obtain a minimax theorem (see Theorem \ref{minimax}). Besides, the assumption of weak compactness (a) on the set $\mathcal{P}$ ensures the ``lower continuity'' of the sublinear expectation $\mathbb{E}^\mathcal{P}[\cdot]$ for the indicator functions of closed sets or continuous functions, i.e., $\be^{\cp}[\vp_n]\downarrow\be^{\cp}[\vp]$, if $\vp_n={\bf 1}_{F_n}$ or $\vp_n\in C_b(\Omega)$ and $\vp_n\downarrow\vp$ (see Lemma 7 and 8, Theorem 31 in Denis et al. \cite{DHP}).

\noindent(iii) We remark that a typical sublinear expectation, the $G$-expectation defined in Peng \cite{Peng2008b,pengbook}, is associated with such a set of probability measures, according to \cite{DHP}.

\noindent(iv) The sublinear expectation admits an unique representation on $\cp_p(\Omega)$, i.e., let $\cp_1,\cp_2\in\cp_p(\Omega)$, if $\be^{\cp_1}[\vp]=\be^{\cp_2}[\vp]$, $\forall \vp\in C_b(\Omega)$, then $\cp_1=\cp_2$. The proof can be easily derived by Corollary \ref{cor1}.
\end{remark}
The main technical tool in this paper is the minimax theorem in \cite{S} as follows:
\begin{theorem}[Corollary 3.3 in \cite{S}] Let $M$ and $N$ be convex spaces, one of which is compact, and $f$ a function on $M\times N$, quasi-concave-convex and upper semi-continuous-lower semi-continuous. Then $\sup\inf f=\inf \sup f$.
\end{theorem}
In the rest of this paper, the following special case of the above theorem will be quoted, where $f$ is a linear and continuous function.
\begin{theorem}{\label{minimax}}
Let $\cp\in \mathcal{P}_1(\Omega)$. Fixing a singleton $\{\mu^*\}\in \mathcal{P}_1(\Omega)$, we have
$$\inf_{\mu\in\cp}\sup_{||\vp||_{\rm{Lip}\leq 1}}\{E_{\mu^*}[\vp]-E_\mu[\vp]\}=\sup_{||\vp||_{\rm{Lip}\leq 1}}\inf_{\mu\in\cp}\{E_{\mu^*}[\vp]-E_\mu[\vp]\}.$$
\end{theorem}

\begin{corollary}{\label{cor1}}
Let $\cp\in\cp_1(\Omega)$. If there exists a singleton $\{\mu^*\}\in\cp_1(\Omega)$ such that
\begin{align}\label{e22}
E_{\mu^*}[\vp]\leq\be^\cp[\vp],\ \ \ \forall \vp\in C_b(\Omega),
\end{align}
then $\mu^*\in\cp$.
\end{corollary}
\begin{proof}
By (b) in Definition \ref{def24}, we can use $\Phi:=\{\vp:||\vp||_{\rm{Lip}\leq 1}\}$ instead of $C_b(\Omega)$ in (\ref{e22}). Then we can rewrite (\ref{e22}) as
$$\sup_{\vp\in\Phi}\inf_{\mu\in\cp}\{E_{\mu^*}[\vp]-E_\mu[\vp]\}=0.$$
By Theorem \ref{minimax}, we have
$$\inf_{\mu\in\cp}\sup_{\vp\in\Phi}\{E_{\mu^*}[\vp]-E_\mu[\vp]\}=0.$$
Since $\cp$ is weakly compact, we can choose $\bar{\mu}\in\mathcal{P}$ such that $\sup_{\vp\in\Phi}\{E_{\mu^*}[\vp]-E_{\bar{\mu}}[\vp]\}=0$, which implies that $\mu^*=\bar{\mu}\in\cp$.

\end{proof}
\section{Main results}
\noindent In this section, we first study the duality formula for the generalized Kantorovich-Rubinstein distance $\mathcal{W}_1$ on $\mathcal{P}_1(\Omega)$ defined in the previous section.
\begin{theorem}[Duality formula for the generalized Kantorovich-Rubinstein distance]\label{ch2t1} Let $\cp_1$, $\cp_2\in \mathcal{P}_1(\Omega)$. Then, we have
$$\mathcal{W}_1(\cp_1,\cp_2)=\sup_{||\vp||_{\rm{Lip}\leq 1}}\{|\be^{\cp_1}[\vp]-\be^{\cp_2}[\vp]|\}.$$
\end{theorem}

\begin{proof}
By the classical duality formula for the Kantorovich-Rubinstein distance and Theorem \ref{minimax}, we have
\begin{align*}
\sup_{\mu\in\cp_1}\inf_{\nu\in\cp_2}W_1(\mu,\nu)&=\sup_{\mu\in\cp_1}\inf_{\nu\in\cp_2}\sup_{||\vp||_{\rm{Lip}\leq 1}}\{E_{\mu}[\vp]-E_{\nu}[\vp]\}\\
&=\sup_{\mu\in\cp_1}\sup_{||\vp||_{\rm{Lip}\leq 1}}\inf_{\nu\in\cp_2}\{E_{\mu}[\vp]-E_{\nu}[\vp]\}\\
&=\sup_{||\vp||_{\rm{Lip}\leq 1}}\sup_{\mu\in\cp_1}\inf_{\nu\in\cp_2}\{E_{\mu}[\vp]-E_{\nu}[\vp]\}\\
&=\sup_{||\vp||_{\rm{Lip}\leq 1}}\{\be^{\cp_1}[\vp]-\be^{\cp_2}[\vp]\}.
\end{align*}
\noindent Similarly, we can obtain
$$\sup_{\nu\in\cp_2}\inf_{\mu\in\cp_1}W_1(\mu,\nu)=\sup_{||\vp||_{\rm{Lip}\leq 1}}\{\be^{\cp_2}[\vp]-\be^{\cp_1}[\vp]\}.$$
\noindent Therefore,
\begin{align*}
\mathcal{W}_1(\cp_1,\cp_2)&=\max\{\sup_{||\vp||_{\rm{Lip}\leq 1}}\{\be^{\cp_1}[\vp]-\be^{\cp_2}[\vp]\}, \sup_{||\vp||_{\rm{Lip}\leq 1}}\{\be^{\cp_2}[\vp]-\be^{\cp_1}[\vp]\}\}\\
&=\sup_{||\vp||_{\rm{Lip}\leq 1}}\{|\be^{\cp_1}[\vp]-\be^{\cp_2}[\vp]|\}.\end{align*}
\end{proof}

It is well-known that Wasserstein distances can metrize weak convergence in the classical probability theory. Correspondingly, we can prove a similar theorem for sublinear expectations. Before presenting this result, we shall discuss some properties of the weak convergence of sublinear expectations.\\

Given a collection of probability measures $\mathcal{P}$, define
$$\overline{\cp}(A):=\sup_{P\in\cp}P(A);\ \underline{\cp}(A):=\inf_{P\in\cp}P(A).$$
Then, we have the following proposition:

\begin{proposition}\label{p-2}
Let $p\geq 1$. Suppose that $\{\cp_n\}_{n=1}^\infty$ and $\cp$ are weakly compact sets of probability measures on $(\Omega, d)$. If $\{\be^{\cp_n}[\cdot]\}_{n=1}^\infty$ converges weakly to $\be^{\cp}[\cdot]$,   then we have\\
\noindent(i) For each closed set $F$, $$\limsup_{\nn}\overline{\cp_n}(F)\leq\overline{\cp}(F).$$
Or equivalently, for each open set $G$,
$$\liminf_{\nn}\underline{\cp_n}(G)\geq \underline{\cp}(G).$$
\noindent(ii) The set $\cp^*=\bigcup_{n=1}^\infty\cp_n$ is tight.
\end{proposition}
\begin{proof}
\noindent(i) For a closed set $F$, there exists a sequence of bounded continuous function $\{\psi_k\}_{k=1}^\infty$ such that $\psi_k\downarrow\mathbf{1}_F$, thus from Theorem 31 in Denis et al. \cite{DHP}, $\be^\cp[\psi_k]\downarrow\overline{\cp}(F)$. On the other hand,
$$\limsup_{\nn}\overline{\cp_n}(F)\leq\limsup_{\nn}\be^{\cp_n}[\psi_k]=\be^{\cp}[\psi_k], \ \ \forall k\in\mathbb{N},$$
which implies that (i) holds.\\

\noindent(ii) Since $\Omega$ is Polish, we can choose a subset $\{\omega_1,\cdots,\omega_n,\cdots\}$ dense in $\Omega$ and for a fixed $m\in \mathbb{N}$, we consider the cover $\{B(\omega_i,\frac{1}{m}):=\{\omega:d(\omega_i,\omega)<\frac{1}{m}\}\}^\infty_{i=1}$.\\

 Fixing $\varepsilon>0$, we first prove that for each $m$, there exists $k_m$ such that
\begin{equation}\underline{\cp_n}\left(\bigcup_{i=1}^{k_m}B(\omega_i,\frac{1}{m})\right)>1-\ve/2^m, \ \ \forall n\in\mathbb{N}.\label{equ1}\end{equation}
Otherwise, there exists $m_0$, such that for each $k$, there exists $n_k$,
$$\underline{\cp_{n_k}}\left(\bigcup_{i=1}^{k}B(\omega_i,\frac{1}{m_0})\right)\leq 1-\ve/2^{m_0}.$$
For each $n\in \mathbb{N}$, from the weak compactness of $\mathcal{P}_n$ and Lemma 8 in \cite{DHP}, we know
%$\bigcup_{i=1}^\infty B(\omega_i,\frac{1}{m_0})=\Omega$,
 $$\underline{\cp_n}\left(\bigcup_{i=1}^k B(\omega_i,\frac{1}{m_0})\right)=1-\overline{\cp_n}\left(\big(\bigcup_{i=1}^k B(\omega_i,\frac{1}{m_0})\big)^c\right)\uparrow 1,\ \ {\rm as}\ k\rightarrow \infty.$$
Thus, we can prove by contradiction that $\lim_{k \rightarrow \infty} n_k=\infty$. Then, by (i), $\forall j\in\mathbb{N}$,
\begin{align*}
\underline{\cp}\left(\bigcup_{i=1}^{j}B(\omega_i,\frac{1}{m_0})\right)&\leq\liminf_{k\rightarrow\infty}\underline{\cp_{n_{k}}}\left(\bigcup_{i=1}^{j}B(\omega_i,\frac{1}{m_0})\right)\\
&\leq \liminf_{k\rightarrow\infty}\underline{\cp_{n_{k}}}\left(\bigcup_{i=1}^{k}B(\omega_i,\frac{1}{m_0})\right)\leq 1-\ve/2^{m_0},
\end{align*}
which implies that
$$
\overline{\cp}\left(\big(\bigcup_{i=1}^{j} B(\omega_i,\frac{1}{m_0})\big)^c\right)\geq \ve/2^{m_0},\ \ \forall j\in\mathbb{N}.
$$
This is a contradiction to the weak compactness of ${\cp}$.\\
%However, $\bigcup_{i=1}^\infty B(\omega_i,\frac{1}{m_0})=\Omega$ and thus from the weak compactness of $\mathcal{P}$ and Lemma 8 in \cite{DHP}, $$\underline{\cp}\left(\bigcup_{i=1}^k B(\omega_i,\frac{1}{m_0})\right)=1-\overline{\cp}\left(\big(\bigcup_{i=1}^k B(\omega_i,\frac{1}{m_0})\big)^c\right)\uparrow 1,$$
 %which is a contradiction.\\

Choosing $k_m$ as in (\ref{equ1}), we can verify that $K=\bigcap_{m=1}^\infty\bigcup_{i=1}^{k_m}\overline{B(\omega_i,\frac{1}{m})}$ is compact.
Then, for each $n$, we have
\begin{align*}
\overline{\cp_n}(K^c)&=\overline{\cp_n}\left(\bigcup_{m=1}^\infty\big(\bigcup_{i=1}^{k_m}\overline{B(\omega_i, \frac{1}{m})}\big)^c\right)\\
&\leq\sum_{m=1}^\infty\overline{\cp_n}\left(\big(\bigcup_{i=1}^{k_m}\overline{B(\omega_i,\frac{1}{m})}\big)^c\right)\\
&=\sum_{m=1}^\infty\left(1-\underline{\cp_n}\left(\bigcup_{i=1}^{k_m}\overline{B(\omega_i,\frac{1}{m})}\right)\right)\\
&<\sum_{m=1}^\infty\ve/2^m=\ve.
\end{align*}
Therefore, the set $\cp^*=\bigcup_{n=1}^\infty\cp_n$ is tight.
\end{proof}
\begin{proposition}\label{p1}
Let $p\geq 1$. Suppose that $\{\cp_n\}_{n=1}^\infty$ and $\cp$ are elements in $\cp_p(\Omega)$. If $\{\cp_n\}_{n=1}^\infty$ satisfies
\begin{equation}\lim_{K\rightarrow\infty}\limsup_{n\rightarrow\infty}\be^{\cp_n}[d(\omega_0,\cdot)^p\mathbf{1}_{\{\omega\in\Omega:\ d(\omega_0,\omega)\geq K\}}(\cdot)]=0,\label{e1}\end{equation}
then the following statements are equivalent:\\
\noindent(i) $\cw_p(\cp_n,\cp)\rightarrow 0.$\\
\noindent(ii) $\cw_1(\cp_n,\cp)\rightarrow 0.$
\end{proposition}
\begin{proof}
We only need to prove that ``(ii) $\Rightarrow$ (i)''. Let $\tilde{d}=\min(d,1)$, and denote by $\tilde{W}_p$ the classical Wasserstein distances defined with $\tilde{d}$ and $\tilde{\cw}_p$ the generalized one defined in terms of $\tilde{W}_p$ (see Definition \ref{d2}). We first show that
$$\cw_p(\cp_n,\cp)\rightarrow 0\Longleftrightarrow \tilde{\cw}_p(\cp_n,\cp)\rightarrow 0,$$
where the necessity is directly deduced from $\cw_p\geq\tilde{\cw}_p$.\\

To prove the sufficiency, we first see that the distance function $d$ is dominated by the sum of three parts, that is, for $K\geq 1$, $\omega$, $\bar{\omega}$ and $\omega_0\in \Omega$:
$$d(\omega,\bar{\omega})\leq d(\omega,\bar{\omega})\wedge K+2d(\omega,\omega_0)\mathbf{1}_{\{d(\omega,\omega_0)\geq\frac{K}{2}\}}+2d(\bar{\omega},\omega_0)\mathbf{1}_{\{d(\bar{\omega},\omega_0)\geq \frac{K}{2}\}}.$$
Then, there exists a constant $C_p$ only depending on $p$ such that
$$d(\omega,\bar{\omega})^p\leq C_p(K^p \tilde{d}(\omega,\bar{\omega})^p+d(\omega,\omega_0)^p\mathbf{1}_{\{d(\omega,\omega_0)\geq\frac{K}{2}\}}+d(\bar{\omega},\omega_0)^p\mathbf{1}_{\{d(\bar{\omega},\omega_0)\geq\frac{K}{2}\}}).$$
By the same argument in the proof of Theorem 7.12 in Villani \cite{villani1}, we have for two probability measures $\mu$ and $\nu$,
\begin{align*}
W_p(\mu,\nu)^p\leq C_pK^p\tilde{W}_p(\mu,\nu)^p&+C_pE_\mu[d(\omega_0,\cdot)^p\mathbf{1}_{\{\omega:\ d(\omega_0,\omega)\geq\frac{K}{2}\}}(\cdot)]\\
&+C_pE_\nu[d(\omega_0,\cdot)^p\mathbf{1}_{\{\omega:\ d(\omega_0,\omega)\geq\frac{K}{2}\}}(\cdot)],
\end{align*}
which implies that
\begin{align*}\cw_p(\cp_n,\cp)^p\leq C_pK^p\tilde{\cw}_p(\cp_n,\cp)^p&+C_p\be^{\cp_n}[d(\omega_0,\cdot)^p\mathbf{1}_{\{\omega:\ d(\omega_0,\omega)\geq\frac{K}{2}\}}(\cdot)]\\&+C_p\be^{\cp}[d(\omega_0,\cdot)^p\mathbf{1}_{\{\omega:\ d(\omega_0,\omega)\geq\frac{K}{2}\}}(\cdot)].\end{align*}
First letting $\nn$ and then letting $K\rightarrow\infty$, we have $\cw_p(\cp_n,\cp)\rightarrow 0$.\\

Since the distance function $\tilde{d}$ is bounded by 1, all the distances $\tilde{W}_p$ are equivalent (see \S 7.1.2 in \cite{villani1}), so does $\tilde{\cw}_p$. Therefore, we have
$$\cw_p(\cp_n,\cp)\rightarrow 0\ \ \Longleftrightarrow\ \ \tilde{\cw}_p(\cp_n,\cp)\rightarrow 0\ \ \Longleftrightarrow\ \ \tilde{\cw}_1(\cp_n,\cp)\rightarrow 0\ \ \Longleftrightarrow\ \ \cw_1(\cp_n,\cp)\rightarrow 0.$$
\end{proof}
\begin{theorem}[Wasserstein distance metrize weak convergence]\label{ch4t2}
Let $p\geq 1$. Suppose $\{\cp_n\}_{n=1}^\infty$ and $\cp$ are elements in $\cp_p(\Omega)$. Then the following statements are equivalent:\\
\noindent(i) $\cw_p(\cp_n,\cp)\rightarrow 0$.\\
\noindent(ii) For each $\vp\in C(\Omega)$ with the growth condition: for some $\omega_0\in\Omega$, $|\vp(\omega)|\leq C(1+d(\omega_0,\omega)^p)$, $\forall \omega\in \Omega$, we have
$$\lim_{\nn}\be^{\cp_n}[\vp]=\be^\cp[\vp].$$
\noindent(iii) $\{\be^{\cp_n}[\cdot]\}_{n=1}^\infty$ converges weakly to $\be^{\cp}[\cdot]$ and (\ref{e1}) holds.
\end{theorem}

%The following lemma is necessary for the proof of the theorem above, which can be deduced by the Hahn-Banach separation theorem.
%\begin{lemma}\label{ch2p1}
%Let $\cp$ be a convex and weakly compact set of probability measures on $(\Omega,d)$. If there exists a probability measure $\mu^*$ on $(\Omega,d)$ such that
%$$E_{\mu^*}[\vp]\leq\be^\cp[\vp], \ \ \forall \vp\in C_b(\Omega),$$
%then $\mu^*\in\cp$.
%\end{lemma}
\begin{proof} First, we prove the equivalence of (ii) and (iii).

 ``(ii) $\Rightarrow$ (iii)": It is easy to see that $\{\be^{\cp_n}[\cdot]\}_{n=1}^\infty$ converges weakly to $\be^{\cp}[\cdot]$, thus we only need to prove (\ref{e1}) holds. Indeed, the function $f_K(\omega): =d(\omega_0,\omega)^p\mathbf{1}_{\{d(\omega_0,\omega)\geq K\}}$ is upper semi-continuous and thus, can be approximated by $\vp_m\downarrow f_K$ with $|\vp_m(\omega)|\leq C(1+d(\omega_0,\omega)^p)$. Then we have
$$\limsup_{\nn}\be^{\cp_n}[f_K]\leq\limsup_{\nn}\be^{\cp_n}[\vp_m]=\be^{\cp}[\vp_m],\ \forall m\in\mathbb{N}.$$
Letting $m\rightarrow \infty$,
$$\limsup_{\nn}\be^{\cp_n}[f_K]\leq\be^{\cp}[f_K].$$
Since $\cp\in\cp_p(\Omega)$, from Definition \ref{def24} (b), we conclude (\ref{e1}).\\

``(iii) $\Rightarrow$ (ii)": Giving $\vp\in C(\Omega)$ with the growth condition $|\vp(\omega)|\leq C(1+d(\omega_0,\omega)^p)$, let
$$\vp_K=(-C(1+K^p)\vee\vp)\wedge C(1+K^p)$$
 and $\psi_K=\vp-\vp_K$. Then,
\begin{align*}
|\be^{\cp_n}[\vp]-\be^{\cp}[\vp]|\leq |\be^{\cp_n}[\vp_K]-\be^{\cp}[\vp_K]|&+|\be^{\cp_n}[\psi_K]|+|\be^{\cp}[\psi_K]|\\
\leq |\be^{\cp_n}[\vp_K]-\be^{\cp}[\vp_K]|&+ C\be^{\cp_n}[d(\omega_0,\cdot)^p\mathbf{1}_{\{\omega:\ d(\omega_0,\omega)\geq K\}}(\cdot)]\\
&+C\be^\cp[d(\omega_0,\cdot)^p\mathbf{1}_{\{\omega:\ d(\omega_0,\omega)\geq K\}}(\cdot)].
\end{align*}
Therefore,
\begin{align*}\limsup_{\nn}|\be^{\cp_n}[\vp]-\be^{\cp}[\vp]|\leq &C\limsup_{\nn}\be^{\cp_n}[d(\omega_0,\cdot)^p\mathbf{1}_{\{\omega:\ d(\omega_0,\omega)\geq K\}}(\cdot)]\\
&+C\be^\cp[d(\omega_0,\cdot)^p\mathbf{1}_{\{\omega:\ d(\omega_0,\omega)\geq K\}}(\cdot)].\end{align*}
Letting $K\rightarrow\infty$, the first term of the right-hand side goes to 0 by (\ref{e1}), the second one goes to 0 since $\cp\in\cp_p(\Omega)$. Thus (ii) holds.\\

Now, we prove ``(i) $\Rightarrow$ (iii)''. If $\mathcal{W}_p(\cp_n,\cp)\rightarrow 0$, then $\cw_1(\cp_n,\cp)\rightarrow 0$. By Theorem \ref{ch2t1}, we have
$$\sup_{||\vp||_{\rm{Lip}\leq 1}}\{|\be^{\cp_n}[\vp]-\be^{\cp}[\vp]|\}\rightarrow 0,$$
which implies that $\lim_{\nn}\be^{\cp_n}[\vp]=\be^\cp[\vp]$ holds for all Lipschitz functions. For each bounded continuous function $\vp$, there exist sequences of bounded Lipschitz functions $\{\underline{\vp}_k\}_{k=1}^\infty$ and $\{\overline{\vp}_k\}_{k=1}^\infty$ such that
$$\underline{\vp}_k\uparrow \vp \ \ \text{and} \ \ \overline{\vp}_k\downarrow\vp.$$
Then, we can deduce
$$\limsup_{\nn}\be^{\cp_n}[\vp]\leq\liminf_{k\rightarrow\infty}\limsup_{\nn}\be^{\cp_n}[\overline{\vp}_k]=\liminf_{k\rightarrow\infty}\be^\cp[\overline{\vp}_k]=\be^\cp[\vp].$$
Similarly, we have
$$\liminf_{\nn}\be^{\cp_n}[\vp]\geq\limsup_{k\rightarrow\infty}\limsup_{\nn}\be^{\cp_n}[\underline{\vp}_k]=\limsup_{k\rightarrow\infty}\be^\cp[\underline{\vp}_k]=\be^\cp[\vp].$$
In conclusion, $\be^{\cp_n}[\cdot]$ converges weakly to $\be^{\cp}[\cdot]$. \\

It only remains to be seen whether (\ref{e1}) holds. Indeed, one can first verify that for each $\overline{\omega}\in \Omega$,
$$d(\omega_0,\omega)\mathbf{1}_{\{d(\omega_0,\omega)\geq K\}}\leq d(\omega_0,\overline{\omega})\mathbf{1}_{\{d(\omega_0,\overline{\omega})\geq \frac{K}{2}\}}+2d(\omega,\overline{\omega}),$$
and immediately obtain
$$d(\omega_0,\omega)^p\mathbf{1}_{\{d(\omega_0,\omega)\geq K\}}\leq 2^{p-1}d(\omega_0,\overline{\omega})^p\mathbf{1}_{\{d(\omega_0,\overline{\omega})\geq \frac{K}{2}\}}+2^{2p-1} d(\omega,\overline{\omega})^p.$$
Then we can deduce that for any $\mu_n\in\cp_n$ and $\mu\in\cp$,
$$E_{\mu_n}[d(\omega_0,\cdot)^p\mathbf{1}_{\{\omega:\ d(\omega_0,\omega)\geq K\}}(\cdot)]\leq 2^{p-1}E_\mu[d(\omega_0,\cdot)^p\mathbf{1}_{\{\omega:\ d(\omega_0,\omega)\geq \frac{K}{2}\}}(\cdot)]+2^{2p-1}W_p(\mu_n,\mu)^p,$$
which implies that
$$\be^{\cp_n}[d(\omega_0,\cdot)^p\mathbf{1}_{\{\omega:\ d(\omega_0,\omega)\geq K\}}(\cdot)]\leq2^{p-1}\be^{\cp}[d(\omega_0,\cdot)^p\mathbf{1}_{\{\omega:\ d(\omega_0,\omega)\geq \frac{K}{2}\}}(\cdot)]+2^{2p-1}\cw_p(\cp_n,\cp)^p,$$
where the first term of the right-hand side goes to 0 since $\cp\in\cp_p(\Omega)$, and the second one vanishes as a result of (i). Consequently, (\ref{e1}) holds. \\

Finally, we show (iii) $\Rightarrow$ (i). Thanks to Proposition \ref{p1}, we only need to prove the case when $p=1$. By the definition of $\mathcal{W}_1$, it suffices to prove that
$$\sup_{\mu_n\in\cp_n}\inf_{\mu\in\cp}W_1(\mu_n,\mu)\rightarrow 0\ \text{and} \ \sup_{\mu\in\cp}\inf_{\mu_n\in\cp_n}W_1(\mu_n, \mu)\rightarrow 0.$$
\noindent (a) Suppose $\sup_{\mu_n\in\cp_n}\inf_{P\in\cp}W_1(\mu_n,\mu)\nrightarrow 0$, then there exists an $\ve>0$ and a subsequence $\{\cp_{n_k}\}_{k=1}^\infty\subset \{\cp_{n}\}_{n=1}^\infty$, such that for each $n_k$, there exists a $\mu_{n_k}\in\cp_{n_k}$ with
\begin{equation}\label{con}
\inf_{\mu\in\cp}W_1(\mu_{n_k},\mu)\geq \ve,
\end{equation}
which implies that for all $\mu\in\cp, W_1(\mu_{n_k},\mu)\geq\ve$.
By Proposition \ref{p-2}, the set $\bigcup_{n=1}^\infty \cp_n$ is tight, so it is relatively weakly compact by Prokhorov's theorem. Then, there exists a subsequence of $\{\mu_{n_k}\}_{k=1}^\infty$, still denoted by $\{\mu_{n_k}\}_{k=1}^\infty$, which converges weakly to a probability measure $\mu^*$. In this case, we can apply the classical result for Wasserstein distance to obtain $W_1(\mu_{n_k},\mu^*)\rightarrow 0$.
\noindent On the other hand, for each $\vp\in C_b(\Omega)$,
$$E_{\mu^*}[\vp]=\lim_{k\rightarrow\infty}E_{\mu_{n_k}}[\vp]\leq\lim_{k\rightarrow\infty}\be^{\cp_{n_k}}[\vp]=\be^{\cp}[\vp],$$
\noindent which implies from Corollary \ref{cor1} that $\mu^*\in\cp$. Therefore, we obtain immediately that
$$\inf_{\mu\in\cp}W_1(\mu_{n_k},\mu)\leq W_1(\mu_{n_k},\mu^*)\rightarrow 0,\ \text{as} \ k\rightarrow\infty,$$
which is in contradiction to (\ref{con}).\\[6pt]
\noindent (b) Suppose $\sup_{\mu\in\cp}\inf_{\mu_n\in\cp_n}W_1(\mu_n, \mu)\nrightarrow 0$, then there exist an $\ve>0$, a sequence $\{\mu_{n_k}^*\}_{k=1}^\infty\subset\cp$ and a subsequence $\{\cp_{n_k}\}_{k=1}^\infty\subset \{\cp_{n}\}_{n=1}^\infty$, such that
$$\inf_{\mu_{n_k}\in\cp_{n_k}}W_1(\mu_{n_k}, \mu_{n_k}^*)\geq4\ve.$$
\noindent Similarly to (a), one can find a subsequence of $\{\mu_{n_k}^*\}_{k=1}^\infty$, still denoted by $\{\mu^*_{n_k}\}_{k=1}^\infty$, which converges weakly to a probability measure $\mu^*\in\cp$. For sufficient large $k$,
$$\inf_{\mu_{n_k}\in\cp_{n_k}}W_1(\mu_{n_k}, \mu^*)\geq 3\ve.$$
\noindent From the classical duality formula for the Kantorovich-Rubinstein distance (\ref{dua}), we have
$$\inf_{\mu_{n_k}\in\cp_{n_k}}W_1(\mu_{n_k}, \mu^*)=\inf_{\mu_{n_k}\in\cp_{n_k}}\sup_{||\vp||_{\text{\rm{Lip}}\leq 1}}\{E_{\mu^*}[\vp]-E_{\mu_{n_k}}[\vp]\}.$$
\noindent Since $\cp_{n_k}$ and $\Phi:=\{\vp:||\vp||_{\rm{Lip}\leq 1}\}$ are convex as well as $\cp_{n_k}$ is weakly compact, the minimax theorem implies
\begin{align*}\inf_{\mu_{n_k}\in\cp_{n_k}}\sup_{||\vp||_{\rm{Lip}\leq 1}}\{E_{\mu^*}[\vp]-E_{\mu_{n_k}}[\vp]\}&=\sup_{||\vp||_{\rm{Lip}\leq 1}}\inf_{\mu_{n_k}\in\cp_{n_k}}\{E_{\mu^*}[\vp]-E_{\mu_{n_k}}[\vp]\}\\
&=\sup_{||\vp||_{\rm{Lip}\leq 1}}\{E_{\mu^*}[\vp]-\be^{\cp_{n_k}}[\vp]\}\geq 3\ve.
\end{align*}
Therefore, thanks to (\ref{e1}), one can find a $\vp\in C_{b}(\Omega)$, such that for sufficient large $k$,
$$E_{\mu^*}[\vp]\geq\be^{\cp_{n_k}}[\vp]+2\ve\geq\be^{\cp}[\vp]+\ve,$$
which is in contradiction to $\mu^*\in\cp$.\hfill$\square$
\end{proof}
\begin{corollary}\label{cor}
Let $p\geq 1$. Suppose that $\cp_n$ and $\cp$ be sets of probability measures on $(\Omega, d)$ satisfying all conditions in Definition \ref{def24} expect the assumption of convexity. If $\cw_p(\cp_n,\cp)\rightarrow 0$, then (ii) and (iii) in Theorem \ref{ch4t2} still hold.
\end{corollary}
\begin{proof}
Let $\cp_n^*$ and $\cp^*$ be the convex hull of $\cp_n$ and $\cp$ respectively. We first observe that $\be^{\cp^*}[\vp]=\be^{\cp}[\vp]$ for all $\vp$, which implies that $\cp_n^*$, $\cp^*\in\cp_p(\Omega)$. As in the proof of above theorem, we also have
\begin{align*}\be^{\cp^*_n}[d(\omega_0,\cdot)^p\mathbf{1}_{\{\omega:\ d(\omega_0,\omega)\geq K\}}(\cdot)]&=\be^{\cp_n}[d(\omega_0,\cdot)^p\mathbf{1}_{\{\omega:\ d(\omega_0,\omega\geq K\}}(\cdot)]\\
&\leq2^{p-1}\be^{\cp}[d(\omega_0,\cdot)^p\mathbf{1}_{\{\omega:\ d(\omega_0,\omega)\geq \frac{K}{2}\}}(\cdot)]+2^{2p-1}\cw_p(\cp_n,\cp)^p,
\end{align*}
thus $\{\cp^*_n\}_{n=1}^\infty$ also satisfies (\ref{e1}).
 So it suffice to prove that $\cw_1(\cp_n^*,\cp^*)\rightarrow 0$, which yields $\cw_p(\cp_n^*,\cp^*)\rightarrow 0$ by Proposition \ref{p1}.\\

Without the assumption of convexity on the sets of probability measures, we still have for each $n$,
\begin{align*}
\sup_{\mu\in\cp_n}\inf_{\nu\in\cp}W_1(\mu,\nu)&=\sup_{\mu\in\cp_n}\inf_{\nu\in\cp}\sup_{||\vp||_{\rm{Lip}\leq 1}}\{E_{\mu}[\vp]-E_{\nu}[\vp]\}\\
&\geq\sup_{\mu\in\cp_n}\sup_{||\vp||_{\rm{Lip}\leq 1}}\inf_{\nu\in\cp}\{E_{\mu}[\vp]-E_{\nu}[\vp]\}\\
&=\sup_{||\vp||_{\rm{Lip}\leq 1}}\sup_{\mu\in\cp_n}\inf_{\nu\in\cp}\{E_{\mu}[\vp]-E_{\nu}[\vp]\}\\
&=\sup_{||\vp||_{\rm{Lip}\leq 1}}\{\be^{\cp_n}[\vp]-\be^{\cp}[\vp]\}.
\end{align*}
Similarly,
$$\sup_{\nu\in\cp}\inf_{\mu\in\cp_n}W_1(\mu,\nu)\geq\sup_{||\vp||_{\rm{Lip}\leq 1}}\{\be^{\cp}[\vp]-\be^{\cp_n}[\vp]\}.$$
\noindent Therefore,
\begin{align*}
\mathcal{W}_1(\cp_n,\cp)&\geq\sup_{||\vp||_{\rm{Lip}\leq 1}}\{|\be^{\cp_n}[\vp]-\be^{\cp}[\vp]|\}\\
&=\sup_{||\vp||_{\rm{Lip}\leq 1}}\{|\be^{\cp^*_n}[\vp]-\be^{\cp^*}[\vp]|\}=\cw_1(\cp^*_n,\cp^*).\end{align*}
From which we can imply $\cw_1(\cp_n^*,\cp^*)\rightarrow 0$.
By Theorem \ref{ch4t2}, (ii) and (iii) hold for $\be^{\cp_n^*}[\cdot]$ and $\be^{\cp^*}[\cdot]$, so do $\be^{\cp_n}[\cdot]$ and $\be^{\cp}[\cdot]$.
\end{proof}
\begin{remark}
In general, without the assumption of convexity, the statement (ii) and (iii) in Theorem \ref{ch4t2} may not imply $\cw_p(\cp_n,\cp)\rightarrow 0$. Here is a simple counterexample: consider a set $\cp$ satisfying Definition \ref{def24} except for the convexity and such that $\mathcal{W}_1(\cp,\cp^*)>0$, where $\cp^*$ is the convex hull of $\cp$.
% and we know .
Let $\cp_n=\cp^*$, for each $n\in \mathbb{N}$. Obviously, (ii) and (iii) in Theorem 3.4 hold since $\be^{\cp_n}[\vp]=\be^{\cp^*}[\vp]=\be^{\cp}[\vp]$, but $\mathcal{W}_1(\cp_n,\cp)\nrightarrow 0$.

%We also give a example to show that $\mathcal{W}_1(\cp^*_1,\cp^*_2)\neq \mathcal{W}_1(\cp_1,\cp_2)$, where $\cp_1$ and $\cp_2$ satisfy the conditions of Corollary \ref{cor}. For instance, consider $\Omega=\br$, and $\delta_x$ is Dirac measure on $\br$. Let
%$$
%\mathcal{P}_1:=\{a\delta_{0}+b\delta_{1}+c\delta_{2}: 0\leq a\leq 0.5, 0\leq c\leq 0.5, a+b+c=1\},
%$$
%and
%\begin{align*}
%\mathcal{P}_2:=\{a\delta_0+(1-a)\delta_1: 0\leq a\leq 0.75\}\bigcup\{(1-a)\delta_1+a\delta_2: 0\leq a\leq 0.75\},
%\end{align*}
%where both sets are weakly compact and only $\mathcal{P}_1$ is convex. Then, one can verify that
%\begin{align*}
%\mathcal{P}^*_2:=\{a\delta_0+b\delta_1+c\delta_2: 0\leq a+c\leq 0.75, a+b+c=1\},
%\end{align*}
%and
%$$0.25= \mathcal{W}_1(\cp^*_1,\cp^*_2)<\mathcal{W}_1(\cp_1,\cp_2)=0.5.$$
\end{remark}
\noindent{\bf{Acknowledgements}} The authors gratefully acknowledge the helpful suggestions and comments from both the associate editor and the anonymous reviewer. X. Li is supported by the China Postdoctoral Science Foundation (No.2014M561907) and the Fundamental Research Funds of Shandong University (No.2014GN007); Y. Lin is supported by the European Research Council (ERC) under grant FA506041 and by the  Austrian  Science Fund  (FWF)  under  grant P25815.\\[6pt]

\end{document}